\documentclass{amsart}
\usepackage{amsmath}
\usepackage{amssymb}
\usepackage[english]{babel}
\usepackage[utf8]{inputenc}
\usepackage{amsfonts}
\usepackage{amsthm}
\usepackage{color}
\usepackage{graphicx}
\usepackage{esint}
\numberwithin{equation}{section}
\renewcommand{\subjclassname}{%
	\textup{2010} Mathematics Subject Classification: }

\newtheorem{thm}{Theorem}[section]
\newtheorem{prop}[thm]{Proposition}

\newtheorem{lem}[thm]{Lemma}
\theoremstyle{remark}
\newtheorem{rmk}[thm]{Remark}
\theoremstyle{definition}
\newtheorem{defn}{Definition}[section]

\DeclareMathOperator{\N}{\mathbb{N}}
\DeclareMathOperator{\R}{\mathbb{R}}
\DeclareMathOperator{\cF}{\mathcal{F}}

\DeclareMathOperator{\cL}{\mathcal{L}}

\DeclareMathOperator{\fM}{\mathfrak{M}}

\DeclareMathOperator{\Lip}{Lip}
\DeclareMathOperator{\diam}{diam}
\DeclareMathOperator{\divg}{div}

\newcommand{\Norm}[2]{\left\Vert #1 \right\Vert_{#2}}

\begin{document}
\title{Atomic decomposition of finite signed measures
	on compacts of $\R^n$}
\author{Francesca Angrisani}\thanks{francesca.angrisani@unina.it}
\author{Giacomo Ascione}\thanks{giacomo.ascione@unina.it}
\author{Gianluigi Manzo}\thanks{gianluigi.manzo@unina.it}
\address{Universit\'a degli Studi di Napoli Federico II, Dipartimento di Matematica e Applicazioni ``Renato Caccioppoli'', Via Cintia, Monte S. Angelo I-80126 Napoli, Italy}
\maketitle
\begin{abstract}
	Recently there has been interest in pairs of Banach spaces $(E_0,E)$ in an $o-O$ relation and with $E_0^{**}=E$. It is known that this can be done for Lipschitz spaces on suitable metric spaces. In this paper we consider the case of a compact subset $K$ of $\R^n$ with the Euclidean metric, which does not give an $o-O$ structure, but we use part of the theory concerning these pairs to find an atomic decomposition of the predual of $Lip(K)$. In particular, since the space $\fM(K)$ of finite signed measures on $K$, when endowed with the Kantorovich-Rubinstein norm, has as dual space $Lip(K)$, we can give an atomic decomposition for this space.\\
\end{abstract}
\begin{small}
\subjclassname{46E15, 46E27, 46E35} 
\end{small} \\
\section{Introduction}
L. Hanin has dedicated some papers \cite{hanin1994isometric,hanin1997duality} to the description of spaces in duality with Lipschitz spaces, namely spaces of finite signed Borel measures on compact metric spaces $K$. His results have been extended to the case of separable metric spaces (non necessarily compact) in \cite{hanin1999extension}. In what follows we will consider $K$ a compact domain in $\mathbb{R}^n$ equipped with the euclidean norm, which we denote here by $|\cdot|$. The choice of a compact domain and the Euclidean distance will be made clear in what follows. More precisely, when we endow the space $\fM(K)$ of such measures on $K$ with the so-called Kantorovich-Rubinstein norm and consider its completion, we obtain a space that is isometric to the predual of the space of Lipschitz functions of $K$.\\
The Kantorovich-Rubinstein norm (see section 2) was introduced in the context of optimal transport theory. As a matter of fact, the distance, induced by the norm, between two measures $\mu$ and $\nu$ with same total mass, i.e. $\mu(K)=\nu(K)$, is simply the cost of the optimal transport from one to the other (see next section for definitions).\\ 
Other than identifying $\fM(K)^*$ as $\Lip(K)$, passing to duals, one can also investigate embedding properties of $\fM(K)^c$, or of $\fM_0(K)^c$, in its bidual $\Lip(K)^*$, where $\fM_0(K)$ is the subspace of $\fM(K)$ containing only measures with null total mass, called balanced measures.\\
An interesting consequence of this approach is that it inspires the introduction of the dual problem in optimal transport theory. As a matter of fact, by thinking of elements in $\fM(K)$ as functionals on $\Lip(K)$ we obtain that the Kantorovich-Rubinstein norm on $\fM_0(K)$ is equal to the norm
$$\|\mu\|_{KR_0}=\sup\left\{\int_K f d\mu, \quad f \in \Lip_1(K) \right\}$$ where $\Lip_1(K)$ is the set of Lipschitz functions with Lipschitz constant $L\le 1$ (see \cite[Remark $6.5$]{villani2008optimal}). The fundamental problem of optimal transport theory, i.e. finding, if it exists, a minimizer to the minimization problem occurring in the definition of $\|\mu-\nu\|_{KR}$ with $\mu$ and $\nu$ measures with same total mass, is then equivalently formulated as a maximization problem. It was proven in \cite{BouchitteCompletion} that elements in $\fM_0(K)^c$ are precisely those for which the dual problem admits maximizers. 
Moreover, in the same paper, the space $\fM_0(K)^c$ is also characterized as the space of the distributional divergences of $L^1(K;\R^n)$ functions.\\
In this paper we will give an atomic decomposition of the spaces $\fM(K)$ and $\fM_0(K)$ by restriction of the decomposition of their completions, seen as preduals of Lipschitz spaces. We recall that the description of atomic decompositions of H\"older spaces on compact spaces was given in \cite{jonsson1990duals} and \cite{angrisani2019duality}, following different approaches; in particular, in \cite{jonsson1990duals} the atomic decomposition is closer to other "classical" examples \cite{dafni2018space,d2020atomic}, while in \cite{angrisani2019duality} a more abstract atomic decomposition is obtained. We decided to follow this second approach, based on techniques from \cite{d2020atomic}, which are inspired by the $o$--$O$ construction in \cite{Perfekt}. In particular, in the case of the distance $d_\alpha(x,y)=|x-y|^\alpha$ on a compact set of $\R^n$, the $o$--$O$ construction has already been shown in \cite{Perfekt}, while in the general framework of doubling compact metric-measure spaces it has been achieved, under some approximation hypotheses, in \cite{angrisani2019duality}, where the atomic decomposition in this case has been already exploited. However, as already stated in \cite{Perfekt}, such approximation hypotheses do not cover the case of the Euclidean distance, that is here covered without making use of the concept of a $o$--$O$ structure.\\
In particular, we will see in the third section that elements of the embedded copy of $\fM_0(K)^c$ in $\Lip(K)^*$ can be thought of as all the infinite sums of the type $$\mu=\sum\limits_{j=1}^{+\infty}  \frac{\delta_{x_j}-\delta_{y_j}}{|x_j-y_j|}\alpha_j\quad  \text{ with }\alpha_j \text{ satisfying } \quad \sum\limits_{j=1}^{+\infty} |\alpha_j|<+\infty$$
and where $\{x_j\}_{j \in \mathbb{N}}$ and $\{y_j\}_{j \in \mathbb{N}}$ are two disjoint countable dense subsets of $K$.
These infinite sums are viewed as bounded linear functionals on Lipschitz functions $f$ in the following way
$$\langle \mu, f \rangle =\sum\limits_{j=1}^{+\infty} \frac{f(x_j)-f(y_j)}{|x_j-y_j|}\alpha_j $$
where the right hand side is finite because $\frac{|f(x_j)-f(y_j)|}{|x_j-y_j|}$ is bounded by the Lipschitz constant of $f$ and $\alpha_j$ is a sequence in $\ell^1$.\\
On the other hand, for some choices of $\alpha_j$, $\mu$ is not a finite signed Borel measure on $K$, even if the sequence of partial sums is a Cauchy sequence in the Kantorovich norm, showing that $\fM_0(K)$ is not complete.
A fourth section of this paper is dedicated to obtain a similar result for $\fM(K)^c$. In such a case, since we are not identifying functions that differ from each other by a constant, the atomic decomposition will be not only expressed as an infinite linear combination of dipoles, but a correction term in form of an atom (i.e. $\delta_{x_j}$) has to be added to each summand.\\
Let us state that such atomic decomposition represent a first step towards obtaining a Schauder basis for $\fM(K)^c$, which is still an open problem (see \cite[Problem $3$]{ambrosio2016linear}). Moreover, atomic decompositions are actually powerful tools to obtain some interesting functional properties. For instance, in \cite{bonsall1991general}, a general atomic decomposition theorem is used to provide a different proof of Banach's closed range theorem, while in \cite{carando2009duality}, such decompositions are used to determine different properties concerning duality and reflexivity of the decomposed spaces. Atomic decomposition are also used to solve some eigenvalue problems, as done in \cite{gasymov2018atomic}, where some second-order ordinary differential equations are solved by using atomic decompositions in $L^p$. Finally, let us also recall that generalization of atomic decompositions (and of frames of Banach spaces) have been provided, as in \cite{jahan2019approximative}.\\
Concerning the atomic decomposition of $\fM_0(K)^c$ and the series representation of the elements of $\fM(K)^c$, we think such series representation could be useful to work in the framework of differential equations on Banach spaces, as for instance the Kolmogorov equations that arise from Stochastic Partial Differential Equations (see, for instance, \cite{hairer2009introduction} for a general introduction on Stochastic PDEs).
\section{Lipschitz spaces, spaces of Borel measures and their completions}
In this section we will introduce the notation concerning the spaces and the norms we will work with. Let us fix a bounded open set $\Omega\subset \R^n$ and let us denote $K=\overline{\Omega}$.
\subsection{Lipschitz spaces and fractional Sobolev spaces}
\begin{defn}
We define the \textbf{Lipschitz spaces}
\begin{equation*}
\Lip(K)=\left\{f:K \to \R: \ \sup_{\substack{(x,y) \in K^2 \\ x \not = y}}\frac{|f(x)-f(y)|}{|x-y|}<+\infty\right\}
\end{equation*}
and
\begin{equation*}
\Lip_0(K)=\Lip(K)/\R,
\end{equation*}
i.e. the Lipschitz space $\Lip(K)$ modulo constant functions.\\
In $\Lip_0(K)$, to simplify the notation, we will identify any function $f:K \to \R$ with its equivalence class. If we endow $\Lip_0(K)$ with the norm 
\begin{equation*}
\Norm{f}{\Lip_0(K)}=\sup_{\substack{(x,y) \in K^2 \\ x \not = y}}\frac{|f(x)-f(y)|}{|x-y|},
\end{equation*}
then this normed space is a Banach space, while on $\Lip(K)$ the functional $\|\cdot\|_{\Lip_0(K)}$ would only work as a seminorm.\\
Furthermore, $\Lip(K)$ would be a Banach space if endowed with the norm
\begin{equation*}
\Norm{f}{\Lip(K)}=\max\{\Norm{f}{\Lip_0(K)},\Norm{f}{L^\infty(K)}\}.
\end{equation*}
\end{defn}
In the following we will need to embed the spaces $\Lip(K)$ and $\Lip_0(K)$ in suitable reflexive Banach spaces. For our purposes, the natural candidates are fractional Sobolev spaces. An almost complete survey on such spaces is given in \cite{di2012hitchhikers}.
\begin{defn}
	Let us denote by $W^{s,p}(\Omega)$ for $s \in (0,1)$ and $p>1$ the fractional Sobolev space consisting of the functions $f \in L^p(\Omega)$ such that
	\begin{equation*}
	\Norm{f}{\dot{W}^{s,p}(\Omega)}^p:=\int_{\Omega}\int_{\Omega}\frac{|f(x)-f(y)|^p}{|x-y|^{ps+n}}dxdy<+\infty.
	\end{equation*}
	If we endow $W^{s,p}(\Omega)$ with the norm
	\begin{equation*}
	\Norm{f}{W^{s,p}(\Omega)}=\Norm{f}{\dot{W}^{s,p}(\Omega)}+\Norm{f}{L^p(\Omega)}
	\end{equation*}
	it is a reflexive separable Banach space (since it is uniformly convex by means of a Clarkson-type inequality \cite{Grigoryan}). The homogeneous fractional Sobolev space $\dot{W}^{s,p}(\Omega)$ is defined as $\dot{W}^{s,p}(\Omega)=W^{s,p}(\Omega)/\R$ and if we endow this space with the norm $\Norm{f}{\dot{W}^{s,p}(\Omega)}$ it is a reflexive separable Banach space (for the same reason as before). 
\end{defn}
\begin{rmk}
	Let us recall that if $ps>n$, by a fractional Morrey-type embedding theorem, we have that $W^{s,p}(\Omega) \hookrightarrow C(K)$ (this is true for any doubling compact metric-measure space as a consequence of the Morrey embedding for Haj\l{}asz-Sobolev spaces \cite[Theorem $8.7$]{hajlasz2003sobolev} and the continuous embedding of Besov spaces in them \cite[Lemma $6.1$]{gogatishvili2010interpolation}). In this case we will always consider the continuous realization of a function in $W^{s,p}(\Omega)$.
\end{rmk}
Another characterization of $\dot{W}^{s,p}(\Omega)$ for $sp>n$ is given as the space of functions $f \in W^{s,p}(\Omega)$ such that $f(z)=0$, for an a priori fixed point $z \in K$ (here we are implicitly using the embedding $W^{s,p}(\Omega) \hookrightarrow C(K)$). In particular we have (by using the same idea adopted for $\Lip(K)$) that the norm
\begin{equation*}
\Norm{f}{W^{s,p}(\Omega),z}=\Norm{f}{\dot{W}^{s,p}(\Omega)}+|f(z)|
\end{equation*}
is equivalent to $\Norm{\cdot}{W^{s,p}(\Omega)}$. By identifying $C(K)/\R$ in the same way we have $\dot{W}^{s,p}(\Omega) \hookrightarrow C(K)/\R$.
\subsection{Spaces of Borel measures}
The definitions and considerations of this section can be also applied for any general metric space. However, here we focus on the Euclidean case, as it is the main scope of the paper.
\begin{defn}
	We denote the space of finite signed Borel measures on $K$ by $\fM(K)$, the subspace of finite positive measures on $K$ by $\fM_+(K)$, and the subspace of $\fM(K)$ consisting only of measures $\mu$ such that $\mu(K)=0$ by $\fM_0(K)$.\\
	Via the Hahn-Jordan decomposition, a signed measure $\mu$ can be seen as the difference of two positive Borel measures $\mu^+$ and $\mu^-$, i.e. $\mu=\mu^+-\mu^-$; the total variation of $\mu$ is defined as the sum of the two, i.e. $|\mu|=\mu^++\mu^-$.
\end{defn}
The total variation $\mu \in \fM(K)\mapsto |\mu|(K) \in \R$ is a norm on $\fM(K)$ that gives to the space the structure of Banach space. However, it does not take into account the metric structure of the domain $K$ (for instance $|\delta_x-\delta_y|(K)=2$, for any $(x,y)\in K^2$ with $x \not = y$). On the other hand, even in the more general setting of a compact metric space $K$, Kantorovich and Rubinstein (see \cite{kantorovich1958space} and \cite{vershik2013long} for a complete historical review) introduced a norm $\|\cdot\|_{KR}$ on $\fM(K)$ inducing a distance that is a natural extension of the distance on $K$.\\
As a matter of fact, $K$ naturally embeds in $\fM(K)$ by associating to each point $x$ in $K$ the Dirac measure $\delta_x$ concentrated in $x$. We will introduce a norm $\|\cdot\|_{KR}$ that will have the interesting property that $\|\delta_x-\delta_y\|_{KR}=\min\{|x-y|,2\}$, in some sense extending the metric on $K$ to $\fM(K)$.\\
To define the Kantorovich-Rubinstein norm on $\fM(K)$, we first start by doing so on the space $\fM_0(K)\subset \fM(K)$ of \textbf{balanced measures} $\mu$, i.e. such that $\mu(K)=0$ and hence $\mu^+(K)=\mu^-(K)$.
\begin{defn}[\cite{kantorovich1942mass,kantorovich1957functional,kantorovich1958space}]
	Consider any $\mu \in \fM_0(K)$ and define a family $\Psi_\mu \subset \fM_+(K \times K)$ of positive Borel measures on the Cartesian square $K\times K$ of $K$ in the following way: $\Psi \in \Psi_\mu$ if and only if, for any Borel set $E \subset K$, $\Psi(K,E)-\Psi(E,K)=\mu(E)$ (called \textbf{balance condition})\\
	The Kantorovich-Rubinstein norm of $\mu$ is defined as
	\begin{equation*}
	\Norm{\mu}{KR_0}:=\inf\left\{\int_{K \times K}|x-y|d\Psi(x,y): \ \Psi \in \Psi_\mu\right\}.
	\end{equation*}
	\end{defn}
\begin{defn}
	For $\mu \in \fM(K)$ we define the ``extended'' Kantorovich-Rubinstein norm (as done in \cite{Hanin}) of $\mu$ as
	\begin{equation*}
	\Norm{\mu}{KR}:=\inf\{\Norm{\nu}{KR_0}+|\mu-\nu|(K): \ \nu \in \fM_0(K)\}.
	\end{equation*}
\end{defn}
An important thing to notice is that $(\fM_0(K),\Norm{\cdot}{KR_0})$ and $(\fM(K),\Norm{\cdot}{KR})$ are not Banach spaces.
\begin{rmk}
	Given $(x,y)\in K$ we have $\Norm{\delta_x-\delta_y}{KR_0}=|x-y|$ while $\Norm{\delta_x}{KR}=1$, showing that the Kantorovich-Rubinstein norm satisfies the desired property of concordance with the metric on $K$.
\end{rmk}
	The completion of the space of finite Borel measure on $K$ with respect to the Kantorovich-Rubinstein norm is denoted by $\fM(K)^c$, while we denote by $\fM_0(K)^c$ the completion of $\fM_0(K)$ with respect to the norm $\Norm{\cdot}{KR_0}$.
 \\
 It has been shown (see for instance \cite{Hanin}) that $\fM(K)^*$ is isometric to $\Lip(K)$ while $\fM_0(K)^*$ is isometric to $\Lip_0(K)$. Moreover, it is interesting to recall a characterization of $\fM_0(K)^c$. Indeed, in \cite{BouchitteCompletion} it is shown that if $K$ is a compact subset of $\R^n$ then for any functional $\mu \in \fM_0(K)^c$ there exists a function $f \in L^1(K;\R^n)$ such that
 \begin{equation*}
 \mu=\divg f.
 \end{equation*}
Moreover (see \cite{BouchitteCompletion}), for any functional $\mu \in \fM_0(K)^c$ there exists a function $g \in B_{\Lip_0(K)}$ (where for any Banach space $X$ we denote by $B_X$ the closed unit ball in $X$) such that
\begin{equation*}
\Norm{\mu}{KR_0}=\langle g, \mu \rangle,
\end{equation*}
so that the norm is attained. Let us remark that last formula holds for any separable metric space.
\section{Atomic decomposition of $\fM_0(K)^c$}
Our aim is to give an atomic decomposition of elements $\mu$ of $\fM_0(K)^c$, and so in particular of measures that are balanced on $K$, i.e. such that $\mu(K)=0$, as an infinite sum of simpler elements that we will call \textit{atoms}.
\begin{defn}
We will call $\delta$-\textbf{atom} any measure $\mu \in \fM(K)$ whose support is finite. Moreover, we call \textbf{dipoles} the measure $\mu \in \fM_0(K)$ of the form $\mu=\alpha(\delta_x-\delta_y)$ for some $\alpha \in \R$ and $(x,y)\in K^2$.	
\end{defn}
To obtain a decomposition of elements of $\fM_0(K)^c$ - which will induce a decomposition of elements of $\fM_0(K)$ - we generalize the approach of \cite{angrisani2019orlicz}, which relies on the $o$--$O$ structure of $(c^{0,\alpha},C^{0,\alpha})$, by using results contained in \cite{d2020atomic}, which allow us to remove the dependence on the "little o" space, because for $\Lip$ and $\Lip_0$ it is trivial. We start by writing $\Lip_0$ in a suitable way. Indeed we want to make use of \cite[Theorem $3$]{d2020atomic} and to do this we have to characterize $\Lip_0$ by means of linear bounded operators $L:X \to Y$ where $X$ is a reflexive Banach space containing $\Lip_0$ and $Y$ is some other Banach space. In particular, we want to find a countable family $\cF=\{L_j\}_{j \in \N}$ of such kind of operators such that $$\Lip_0(K)=\{f \in X: \ \sup_{j \in \N}\Norm{L_jf}{Y}<+\infty\}$$. As we will se from the following Lemma, the natural choice we have for $Y$ is $\R$ and for $X$ is $\dot{W}^{s,p}(K)$. Indeed, as we stated before, $\dot{W}^{s,p}(K)$ is separable and reflexive and contains $\Lip_0(K)$ by definition. Moreover, we can chose $s$ and $p$ in a suitable way to obtain $\dot{W}^{s,p}(K)$ continuously embedded in the quotient space $C(K)/\R$. This choice will be useful to show the boundedness of $L_j$. Here the compactness of $K$ plays a prominent role, since in such case $C(K)\subset L^\infty(K)$ (that will be important to show boundedness of $L_j$). In case we choose $K$ to be not compact (for instance unbounded), then we need to find a different approach to show boundedness of the operators. By now, let us focus on the compact case.
\begin{lem}\label{lem0}
	There exists a sequence of functionals $(L_j)_{j \in \N}: X=(\dot{W}^{s,p}(\Omega))\to Y=\R$ such that
	\begin{equation*}
	\Lip_0(K)=\{f \in \dot{W}^{s,p}(\Omega): \ \sup_{j \in \N}|L_jf|<+\infty\}
	\end{equation*}
	and
	\begin{equation*}
	\Norm{f}{\Lip_0(K)}=\sup_{j \in \N}|L_jf|.
	\end{equation*}
\end{lem}
\begin{proof}
	First of all, let us fix $s \in (0,1)$ and $p>1$ such that $ps>n$, so that $\dot{W}^{s,p}(\Omega)\hookrightarrow C(K)/\R$. Let us consider $D_1 \subset K$ a countable set such that $K=\overline{D}_1$ and $K_1=K \setminus D_1$. Now let us consider $D_2 \subset K_1$ a countable set such that $K_1=\overline{D_2}$. Finally, let us define $D=D_1 \times D_2$. Observe that $D_1 \cap D_2=\emptyset$ so, for any $(x,y)\in D$, $x \not = y$. Moreover, $D$ is countable, hence we can enumerate $D=\{(x_j,y_j)\}_{j \in \N}$. Finally $\overline{D}=K \times K$. Let us define
	\begin{equation*}
	L_j: f \in \dot{W}^{s,p}(\Omega) \to \frac{f(x_j)-f(y_j)}{|x_j-y_j|} \in \R.
	\end{equation*}
	$L_j$ is obviously linear. Moreover, since $\dot{W}^{s,p}(\Omega) \hookrightarrow C(K)/\R$ we have
	\begin{equation*}
	\frac{f(x_j)-f(y_j)}{|x_j-y_j|}\le\frac{2}{|x_j-y_j|}\Norm{f}{L^\infty(K)}\le C_j \Norm{f}{\dot{W}^{s,p}(\Omega)},
	\end{equation*}
	hence $L_j \in (\dot{W}^{s,p}(\Omega))^*$ for any $j \in \N$.\\
	Finally, let us observe that by density of $D$ in $K \times K$ and continuity of $f \in \dot{W}^{s,p}(\Omega)$ it holds
	\begin{equation*}
	\Norm{f}{\Lip_0(K)}=\sup_{j \in \N}|L_jf|
	\end{equation*}
	concluding the proof.
\end{proof}
Now that we have this rewriting of the definition of $\Lip_0(K)$ we can use the techniques employed in \cite{d2020atomic} to obtain the desired atomic decomposition. Before giving the main result, let us make use of the ideas behind \cite{d2020atomic}. Indeed, in such case, one can define the operator $V: \Lip_0 \to \ell^\infty$ as, for any $f \in \Lip_0$, $Vf(j)=L_jf$ for any $j \in \N$. Thus, after obtaining that $V\Lip_0 \simeq \Lip_0$ (here we are using $Y=\R$ and $\R^{**} \simeq \R$) it is not difficult to check that the predual $(\Lip_0)_*$ is equivalent to $\ell^1/P$ where $P=(V\Lip_0)^\perp \cap \ell^1$ (where with $\perp$ we denote the annihilator). This gives us a series representation of the elements of the predual of $\Lip_0(K)$, which is actually $\fM_0(K)^c$. This is the starting point of the following result.
\begin{thm}\label{dec0}
	There exists a constant $C \in (0,1)$ such that for any functional $\mu \in \fM_0(K)^c$ there exists a sequence $(\alpha_j)_{j \in \N}\in \ell^1(\R)$ such that
	\begin{equation*}
	\mu=\sum_{j=1}^{+\infty}\frac{\delta_{x_j}-\delta_{y_j}}{|x_j-y_j|}\alpha_j,
	\end{equation*}
	where the series converges in $KR_0$, and
	\begin{equation}\label{bound}
	C\sum_{j=1}^{+\infty}|\alpha_j| \le \Norm{\mu}{KR_0} \le \sum_{j=1}^{+\infty}|\alpha_j|,
	\end{equation}
	where the sequences $(x_j)_{j \in \N}$ and $(y_j)_{j \in \N}$ are defined in Lemma \ref{lem0}. Moreover, the sequence of $\delta$-atoms $(\mu_j)_{j \in \N} \subset \fM_0(K)$ defined as
	\begin{equation*}
	\mu_j=\frac{\delta_{x_j}-\delta_{y_j}}{|x_j-y_j|}
	\end{equation*}
	spans $\fM_0(K)^c$, with $\Norm{\mu_j}{KR_0}=1$ for any $j \in \N$. In particular the $\delta$-atoms $\mu_j$ are dipoles, hence admit support of cardinality exactly $2$.
\end{thm}
\begin{proof}
	By \cite[Theorem $3$]{d2020atomic} we know that there exists $C \in (0,1)$ such that for any $\mu \in \fM_0(K)^c$ there exists a sequence $(\alpha_j)_{j \in \N}$ such that
	\begin{equation*}
	\mu=\sum_{j=1}^{+\infty}L_j^*\alpha_j,
	\end{equation*}
	where $L_j^*$ is the adjoint operator of $L_j$, and
	\begin{equation*}
	C\sum_{j=1}^{+\infty}\Norm{L_j^*\alpha_j}{KR_0}\le \Norm{\mu}{KR_0}\le \sum_{j=1}^{+\infty}\Norm{L_j^*\alpha_j}{KR_0}.
	\end{equation*}
	Since one has
	\begin{equation*}
	\langle f, L_j^*\alpha_j\rangle=\langle L_jf, \alpha_j\rangle=\frac{f(x_j)-f(y_j)}{|x_j-y_j|}\alpha_j,
	\end{equation*}
	then
	\begin{equation*}
	L_j^*\alpha_j=\frac{\delta_{x_j}-\delta_{y_j}}{|x_j-y_j|}\alpha_j
	\end{equation*}
	concluding the proof.
\end{proof}
\begin{rmk}\label{rmkoth}
	Let us remark that one could use any separable Banach space $X$ such that $\Lip(K) \subset X \subset L^\infty(K)$, where the second inclusion is continuous, in place of $W^{s,p}(\Omega)$.\\
	Moreover, let us observe that the previous Theorem provides a $\ell^1/P$-atomic decomposition of $\fM_0(K)^c$.
\end{rmk}
The problem of characterizing the space $\fM_0(K)^c$ has been approached in several ways. In particular it is interesting to remember that in \cite{BouchitteCompletion}, such a space is shown to be isometric to the space $L^1(K;\R^d)/V_0$ where $V_0=\{\sigma \in L^1(K;\R^d): \ \divg\sigma=0\}$, given by $\sigma \in L^1(K;\R^d)/V_0 \mapsto -\divg \sigma \in \fM_0(K)^c$. The motivation of such research towards a characterization of $\fM_0(K)^c$ is linked (as the authors state in the introduction of their paper) to the convergence of infinite sums of dipoles to functionals that are not represented by balanced measures. Here we have shown that such infinite sums of dipoles are indeed all the elements of $\fM_0(K)^c$ and the dipoles represent an atomic part of such a space. Let us finally recall that the infinite sums of dipoles are shown to have a characterization as $-\divg \sigma$ for some $\sigma \in L^1(K;\R^d)$ by using the theory of tangential measures (see \cite[Example $3.7$]{BouchitteCompletion}). In particular this decomposition could be used to determine properties of distributional solutions of partial differential equations involving divergences of $L^1$ functions.
\section{Atomic decomposition of $\fM(K)^c$}
This section is devoted to a similar atomic decomposition in the larger space $\fM(K)^c$, with the help of the space $\Lip(K)$. This time we cannot use the same operators as in Lemma \ref{lem0} since they define a seminorm on $\Lip(K)$. The following rewriting of $\Lip(K)$ relies on the fact that we can consider on $\R^2$ the $\ell^\infty$ norm. 
\begin{lem}\label{lem1}
	There exists a sequence of operators $(L_j)_{j \in \N} \in \cL(W^{s,p}(\Omega),\R^2)$, where we equip $\R^2$ with the norm $\Norm{(x,y)}{\ell^\infty}=\max\{|x|,|y|\}$, such that
	\begin{equation*}
	\Lip(K)=\{f \in W^{s,p}(\Omega): \ \sup_{j \in \N}\Norm{L_jf}{\ell^\infty}<+\infty\}
	\end{equation*}
	and
	\begin{equation*}
	\Norm{f}{\Lip(K)}=\sup_{j \in \N}\Norm{L_jf}{\ell^\infty}.
	\end{equation*}
\end{lem}
\begin{proof}
	First of all, let us fix $s \in (0,1)$ and $p>1$ such that $ps>n$, so that $W^{s,p}(\Omega)\hookrightarrow C(K)$, and let us consider the set $D \subset K^2$ defined in Lemma \ref{lem0}. Let us define
	\begin{equation*}
	L_j: f \in W^{s,p}(\Omega) \to \left(\frac{f(x_j)-f(y_j)}{|x_j-y_j|},f(x_j)\right) \in \R^2.
	\end{equation*}
	$L_j$ is obviously linear. Moreover, since $W^{s,p}(\Omega) \hookrightarrow C(K)$ we have
	\begin{equation*}
	\max\left\{\frac{|f(x_j)-f(y_j)|}{|x_j-y_j|},|f(x_j)|\right\}\le\max\left\{\frac{2}{|x_j-y_j|},1\right\}\Norm{f}{L^\infty(K)}\le C_j \Norm{f}{W^{s,p}(\Omega)},
	\end{equation*}
	hence $L_j \in \cL(W^{s,p}(\Omega),\R^2)$ for any $j \in \N$.\\
	Finally, let us observe that by density of $D$ in $K \times K$, $D_1$ in $K$, and continuity of $f \in \dot{W}^{s,p}(\Omega)$ it holds
	\begin{equation*}
	\Norm{f}{\Lip(K)}=\sup_{j \in \N}\Norm{L_jf}{\ell^\infty}
	\end{equation*}
	concluding the proof.
\end{proof}
As we did in the previous section, we can now use the techniques of \cite{d2020atomic} to obtain the atomic decomposition of $\fM(K)^c$. Let us recall that the starting point of the following result is still the series decomposition that follows from \cite[Theorem $3$]{d2020atomic} that we discussed before Theorem $3.2$. Moreover, let us recall that Remark \ref{rmkoth} holds also for this Theorem.
\begin{thm}
	There exists a constant $C \in (0,1)$ such that for any functional $\mu \in \fM(K)^c$ there exists a sequence $((\alpha_j^1,\alpha_j^2))_{j \in \N}\in \ell^1(\R^2)$ such that
	\begin{equation*}
	\mu=\sum_{j=1}^{+\infty}\left(\frac{\delta_{x_j}-\delta_{y_j}}{|x_j-y_j|}\alpha_j^1+\delta_{x_j}\alpha_j^2\right),
	\end{equation*}
	where the series converges in $KR$, and
	\begin{equation}\label{targetbound}
	C\sum_{j=1}^{+\infty}(|\alpha_j^1|+|\alpha_j^2|)\le \Norm{\mu}{KR}\le\sum_{j=1}^{+\infty}(|\alpha_j^1|+|\alpha_j^2|),
	\end{equation}
	where the sequences $(x_j)_{j \in \N}$ and $(y_j)_{j \in \N}$ are defined in Lemma \ref{lem1}. In particular, the sequence of $\delta$-atoms $(\mu_j)_{j \in \N} \subset \fM(K)$ defined as
	\begin{equation}\label{deltaatoms}
	\mu_j=\begin{cases} \frac{\delta_{x_k}-\delta_{y_k}}{|x_k-y_k|} & j=2k-1\\
	\delta_{x_k} & j=2k
	\end{cases}
	\end{equation}
	 spans $\fM(K)^c$, and $\Norm{\mu_j}{KR}\le 1$ for any $j \in \N$. 
\end{thm}
\begin{proof}
	By \cite[Theorem $3$]{d2020atomic} we know that there exist $\widetilde{C} \in (0,1)$ and $((a_j^1,a_j^2))_{j \in \N}\in \ell^1(\R^2)$ such that for any $\mu \in \fM(K)^c$
	\begin{equation*}
	\mu=\sum_{j=1}^{+\infty}L_j^*\alpha_j,
	\end{equation*}
	where $L_j^*$ is the adjoint operator of $L_j$, $\alpha_j=(\alpha_j^1,\alpha_j^2)\in \R^2$, and
	\begin{equation}\label{prebound}
	\widetilde{C}\sum_{j=1}^{+\infty}\Norm{L_j^*\alpha_j}{KR}\le \Norm{\mu}{KR}\le \sum_{j=1}^{+\infty}\Norm{L_j^*\alpha_j}{KR}.
	\end{equation}
	As in the proof of Theorem \ref{dec0}, we have
	\begin{equation*}
	L_j^*\alpha_j=\frac{\delta_{x_j}-\delta_{y_j}}{|x_j-y_j|}\alpha_j^1+\delta_{x_j}\alpha_j^2.
	\end{equation*}
	Now let us determine some upper and lower bounds for $\Norm{L^*_j \alpha_j}{KR}$. To do this, let us recall that
	\begin{equation*}
	\Norm{\delta_x-\delta_y}{KR}=\min\{|x-y|,2\}\le |x-y|, \qquad \Norm{\delta_x}{KR}=1 \ \forall x,y \in K.
	\end{equation*}
	Hence we have for the upper bound
	\begin{equation}\label{upb}
	\Norm{L^*_j \alpha_j}{KR}\le \frac{\Norm{\delta_{x_j}-\delta_{y_j}}{KR}}{|x_j-y_j|}|\alpha_j^1|+\Norm{\delta_{x_j}}{KR}|\alpha_j^2| \le |\alpha_j^1|+|\alpha_j^2|.
	\end{equation}
	Concerning the lower bound, let us recall (see \cite[Section $4.1$]{hanin1999extension}) that it holds
	\begin{equation}\label{dualrep}
	\Norm{L^*_j \alpha_j}{KR}=\sup_{\Norm{f}{\Lip(K)}\le 1}\left(\frac{f(x_j)-f(y_j)}{|x_j-y_j|}\alpha_j^1+f(x_j)\alpha_j^2\right).
	\end{equation}
	Let $d=\diam(K)$ and let us define the functions
	\begin{equation*}
	f_j(z;\alpha_j^1,\alpha_j^2)=\begin{cases}
	\frac{1-|x_j-z|}{d+1} & \alpha_j^1,\alpha_j^2 \ge 0\\
	\frac{1+|x_j-z|}{d+1} & \alpha_j^1<0\mbox{ and }\alpha_j^2 \ge 0\\
	\frac{-1-|x_j-z|}{d+1} & \alpha_j^1\ge 0\mbox{ and }\alpha_j^2 < 0\\
	\frac{-1+|x_j-z|}{d+1} & \alpha_j^1,\alpha_j^2<0.
	\end{cases}
\end{equation*}
By using this function as test function in \eqref{dualrep} we obtain
\begin{equation}\label{lob}
\Norm{L^*_j \alpha_j}{KR}\ge \frac{1}{d+1}(|\alpha_j^1|+|\alpha_j^2|)
\end{equation}
Using Equations \eqref{upb} and \eqref{lob} in Equation \eqref{prebound} and setting $C=\frac{\widetilde{C}}{d+1}$ we finally achieve Equation \eqref{targetbound}.
\end{proof}
\begin{rmk}
	Let us observe that the sequence of $\delta$-atoms $(\mu_j)_{j \in \N}$ is composed by delta measures and dipoles. In particular if $j$ is even, then $\mu_j$ is a delta measure and then the cardinality of its support is exactly $1$. On the other hand, if $j$ is odd, then $\mu_j$ is a dipole and then the cardinality of its support is exactly $2$. Thus we have that for any functional $\mu \in \fM(K)^c$ there exists a sequence $(\alpha_j)_{j \in \N}\in \ell^1(\R)$ such that $\mu=\sum_{j=1}^{+\infty}\alpha_j\mu_j$ where $\mu_j$ are $\delta$-atoms with support of cardinality at most $2$. \\
	We still have a $\ell^1/P$-atomic decomposition of $\fM(K)^c$. However, in this case, the atoms $\mu_j$ are such that $\Norm{\mu_j}{KR}\le 1$. In particular, if $\diam K\le 2$, we obtain again $\Norm{\mu_j}{KR}=1$ for any $j \in \N$, while, in general, this is true only for even $j$. Let us also observe that to obtain the lower bound in this case, Kantorovich-Rubinstein duality for the norm on $\fM(K)^c$ (see \cite{hanin1999extension}) is actually the main tool.
\end{rmk}
\begin{rmk}\label{rmkinf}
	Let us stress that both inequalities \eqref{bound} and \eqref{targetbound} hold true for respectively a certain sequence $(\alpha_j)_{j \in \N}\in \ell^1(\R)$ and $((\alpha_j^1,\alpha_j^2))_{j \in \N}\in \ell^1(\R^2)$. In particular, setting $\mu \in \fM_0(K)^c$, inequality \eqref{bound} is not necessarily valid for any sequence $(\alpha_j)_{j \in \N}\in\ell^1(\R)$ such that $\mu=\sum_{j=1}^{+\infty}\frac{\delta_{x_j}-\delta_{y_j}}{|x_j-y_j|}\alpha_j$ in $KR_0$. The same holds for \eqref{targetbound}.
\end{rmk}
\begin{rmk}
		Let us observe that if $\mu \in \fM_0(K)^c$ and \linebreak $$\mu=\sum_{j=1}^{+\infty}\left(\frac{\delta_{x_j}-\delta_{y_j}}{|x_j-y_j|}\alpha_j^1+\delta_{x_j}\alpha_j^2\right),$$then $\sum_{j=1}^{+\infty}\alpha_j^2=0$. This is a direct consequence of the fact that $\mu(K)=0$.
\end{rmk}
A similar property holds for any $\mu \in \fM(K)^c$, as we can see from the following Proposition.
\begin{prop}
		Let $\mu \in \fM(K)^c$ and $((\alpha_j^1,\alpha_j^2))_{j \in \N}\in \ell^1(\R^2)$ be the sequence defined in Theorem \ref{dec0}. Suppose $((\beta_j^1,\beta_j^2))_{j \in \N}\in \ell^1(\R^2)$ is another sequence such that
		\begin{equation*}
		\mu=\sum_{j=1}^{+\infty}\left(\frac{\delta_{x_j}-\delta_{y_j}}{|x_j-y_j|}\beta_j^1+\delta_{x_j}\beta_j^2\right)
		\end{equation*}
		and inequalities \eqref{targetbound} hold. Then
		\begin{equation*}
		\sum_{j=1}^{+\infty}(\alpha_j^2-\beta_j^2)=0
		\end{equation*}
	\end{prop}
	\begin{proof}
		Let us define the following measures for $N \in \N$:
		\begin{align*}
		\mu^\alpha_N&=\sum_{j=1}^{N}\frac{\delta_{x_j}-\delta_{y_j}}{|x_j-y_j|}\alpha_j^1+\delta_{x_j}\alpha_j^2\\
		\mu^\beta_N&=\sum_{j=1}^{N}\frac{\delta_{x_j}-\delta_{y_j}}{|x_j-y_j|}\beta_j^1+\delta_{x_j}\beta_j^2\\
		\nu_N&=\mu^\alpha_N-\mu^\beta_N=\sum_{j=1}^{N}\frac{\delta_{x_j}-\delta_{y_j}}{|x_j-y_j|}(\alpha^1_j-\beta^1_j)+\delta_{x_j}(\alpha_j^2-\beta_j^2).
		\end{align*}
		First of all, let us observe that both $\mu^\alpha_N$ and $\mu^\beta_N$ converge in $KR$ norm towards $\mu$.\\
		Now let us observe that
		\begin{align*}
		\sum_{j=1}^{N}\Norm{\frac{\delta_{x_y}-\delta_{y_j}}{|x_j-y_j|}(\alpha^1_j-\beta^1_j)+\delta_{x_j}(\alpha_j^2-\beta_j^2)}{KR}&\le \sum_{j=1}^{N}(|\alpha^1_j-\beta^1_j|+|\alpha^2_j-\beta^2_j|)\\&\le \sum_{j=1}^{N}(|\alpha^1_j|+|\alpha^j_2|)+\sum_{j=1}^{N}(|\beta_j^1|+|\beta_j^2|).
		\end{align*}
		Taking the limit as $N \to +\infty$ we obtain that the series in the left-hand side converges and in particular
		\begin{equation*}
		\sum_{j=1}^{+\infty}\Norm{\frac{\delta_{x_y}-\delta_{y_j}}{|x_j-y_j|}(\alpha_j-\beta_j)}{KR}\le \frac{2}{C}\Norm{\mu}{KR}.
		\end{equation*}
		Now let us consider $M>N>0$ in $\N$ and observe that
		\begin{align*}
		\Norm{\nu_N-\nu_M}{KR}&=\Norm{\sum_{j=N+1}^{M}\frac{\delta_{x_y}-\delta_{y_j}}{|x_j-y_j|}(\alpha^1_j-\beta^1_j)+\delta_{x_j}(\alpha^2_j-\beta^2_j)}{KR}\\
		&\le \sum_{j=N+1}^{M}\Norm{\frac{\delta_{x_y}-\delta_{y_j}}{|x_j-y_j|}(\alpha^1_j-\beta^1_j)+\delta_{x_j}(\alpha^2_j-\beta^2_j)}{KR}.
		\end{align*}
		In particular $(\nu_N)_{N \ge 0}$ is a Cauchy sequence in the Banach space $\fM(K)^c$, thus it admits a limit $\nu \in \fM(K)^c$ given by
		\begin{equation*}
		\nu=\sum_{j=1}^{+\infty}\left(\frac{\delta_{x_y}-\delta_{y_j}}{|x_j-y_j|}(\alpha^1_j-\beta^1_j)+\delta_{x_j}(\alpha^2_j-\beta^2_j)\right).
		\end{equation*}
		Now we need to identify $\nu$. To do this, let us just observe that
		\begin{equation*}
		\nu=\lim_{N \to +\infty}\nu_N=\lim_{N\to +\infty}(\mu^\alpha_N-\mu^\beta_N)=\mu-\mu=0,
		\end{equation*}
		and then we have
		\begin{equation}\label{zerorep}
		0=\sum_{j=1}^{+\infty}\left(\frac{\delta_{x_y}-\delta_{y_j}}{|x_j-y_j|}(\alpha^1_j-\beta^1_j)+\delta_{x_j}(\alpha^2_j-\beta^2_j)\right).
		\end{equation}
		However, we have, by \cite[Equation $1.18$]{hanin1999extension}
		\begin{equation*}
		0=\Norm{0}{KR}=\Norm{\sum_{j=1}^{+\infty}\left(\frac{\delta_{x_y}-\delta_{y_j}}{|x_j-y_j|}(\alpha^1_j-\beta^1_j)+\delta_{x_j}(\alpha^2_j-\beta^2_j)\right)}{KR}\ge \left|\sum_{j=1}^{+\infty}(\alpha^2_j-\beta^2_j)\right|,
		\end{equation*}
		concluding the proof.
	\end{proof}
Let us observe that the same strategy does not lead to uniqueness of the coefficients. Indeed Equation \eqref{zerorep} does not imply $\sum_{j=1}^{+\infty}(|\alpha_j^1-\beta_j^1|+|\alpha_j^2-\beta_j^2|)=0$, in view of Remark \ref{rmkinf}.

\subsection*{Acknowledgements}
We thank the referee for his/her precious advices and critiques, as they showed the way for a cleaner and more efficient exposition of the topic.
\bibliographystyle{plain}
\bibliography{biblip}
\end{document}